\documentclass[12pt]{amsart}
\usepackage{amsfonts,amssymb,amscd,amsmath,enumerate,verbatim}
\usepackage[latin1]{inputenc}
\usepackage{amscd}
\usepackage{latexsym}
\usepackage{enumerate, amsmath, amsthm, amsfonts, amssymb, xy}
\usepackage[margin=1in]{geometry}

\input xy
\xyoption{all}
\usepackage{subfigure}

\newtheorem{theorem}{Theorem}[section]

\newtheorem{definition}[theorem]{Definition}
\newtheorem{example}[theorem]{Example}
\newtheorem{rmk}[theorem]{Remark}
\newtheorem{corollary}[theorem]{Corollary}

\newtheorem{question}[theorem]{Question}

\numberwithin{equation}{section}
\newenvironment{acknowledgement}{Acknowledgment:}

\newcommand{\HH}{\mathcal{H}}

\renewcommand{\phi}{\varphi}
\renewcommand{\emptyset}{\varnothing}

\makeatletter
\def\Ddots{\mathinner{\mkern1mu\raise\p@
\vbox{\kern7\p@\hbox{.}}\mkern2mu
\raise4\p@\hbox{.}\mkern2mu\raise7\p@\hbox{.}\mkern1mu}}
\makeatother

\begin{document}
\title{The face vector of a half-open hypersimplex}
\author{Takayuki Hibi}
\author{Nan Li}
\thanks{Corresponding author: Nan Li: {\tt nan@math.mit.edu}}
\author{Hidefumi Ohsugi}
\address{Takayuki Hibi,
Department of Pure and Applied Mathematics,
Graduate School of Information Science and Technology,
Osaka University,
Toyonaka, Osaka 560-0043, Japan}
\email{hibi@math.sci.osaka-u.ac.jp}
\address{Nan Li,
Department of mathematics,
Massachusetts Institute of Technology,
Cambridge, MA 02139, USA}
\email{nan@math.mit.edu}
\address{Hidefumi Ohsugi,
Department of 
Mathematical Sciences,
School of Science and Technology,
Kwansei Gakuin University,
Sanda, Hyogo, 669-1337, Japan} 
\email{ohsugi@kwansei.ac.jp}

\maketitle

\begin{abstract}
The half-open hypersimplex $\Delta'_{n,k}$ consists of those
$(x_{1}, \ldots, x_{n}) \in[0,1]^n$ with
$k-1<x_1+\cdots+x_n\le k$, where $0 < k \leq n$.
The $f$-vector of a half-open hypersimplex and related generating functions
are explicitly studied.  
\end{abstract}

\section*{Introduction}
A hypersimplex is one of the most basic polytopes and has been well studied. In \cite{Sta}, 
Stanley gave a geometric proof that the normalized volume of the hypersimplex is the Eulerian number. De Loera, Sturmfels and Thomas \cite{DST} studied 
a natural connection of the hypersimplex with Gr\"{o}bner bases. Lam and Postnikov \cite{LP} studied four triangulations of the hypersimplex and showed 
that these triangulations are identical. Other people have also looked at the Ehrhart $h^*$-vectors of the hypersimplex, for example \cite{Kat} and \cite{L}.
The half-open hypersimplex is introduced in \cite{L}, where Li proved a conjecture of Stanley on nice combinatorial description of the Ehrhart $h^*$-vectors.
In this paper, we study the $f$-vectors of the hypersimplex and the half-open hypersimplex.

Let $k$ and $n$ be integers with $0 < k \leq n$.
The {\em hypersimplex} $\Delta_{n,k}$ and 
the {\em half-open hypersimplex} $\Delta'_{n,k}$
are defined as follows:
\[
\Delta_{n,k}=\{ \,  (x_{1}, \ldots, x_{n}) \in[0,1]^n \, : \,  k-1 \leq x_1+\cdots+x_n \le k \, \},
\]
\[
\Delta'_{n,k}=\{ \,  (x_{1}, \ldots, x_{n}) \in[0,1]^n \, : \,  k-1<x_1+\cdots+x_n\le k \, \}.
\]
Let 
$f_{j} (\Delta_{n,k})$ 
denote the number of $j$-faces of $\Delta_{n,k}$
and 
$f_{j}(\Delta'_{n,k})$
those of $\Delta'_{n,k}$, where  
$j = 0, 1, \ldots, n$.  The {\em $f$-vector} of $\Delta_{n,k}$ is
$f(\Delta_{n,k}) = (f_{0}(\Delta_{n,k}), f_{1}(\Delta_{n,k}), \ldots, f_{n}(\Delta_{n,k}))$ 
and that of $\Delta'_{n,k}$ is $f(\Delta'_{n,k})=(f_{0}(\Delta'_{n,k}), f_{1}(\Delta'_{n,k}), \ldots, f_{n}(\Delta'_{n,k}))$.
The computation of the $f$-vector of $\Delta_{n,k}$ is discussed in 
\cite[Exercise 38]{Ziegler}.
In the present paper we are interested in the $f$-vector of $\Delta'_{n,k}$.

First, in Section $1$, a formula to compute $f(\Delta'_{n,k})$ 
is obtained (Theorem \ref{f}).  The formula yields easily a formula to compute 
$f(\Delta_{n,k})$ (Corollary \ref{f-}).
Section $2$ is devoted to the study of generating functions related to $f(\Delta'_{n,k})$
and $f(\Delta_{n,k})$.
More precisely, we show that
\[
\sum_{n\geq1}
\sum_{k=1}^n
\sum_{j=1}^n
f_{j}(\Delta'_{n,k})\ 
x^k y^{n-k} t^j
=
\frac{(1-x)xt}{(1-x-y)(1-x-y-xt)(1-x-y-yt)},
\]
\[
\sum_{n\geq 1}
\sum_{k=1}^n
\sum_{j=1}^n
f_{j}(\Delta_{n,k})\ 
x^k y^{n-k}t^j
=
\frac{xt}{(1-x-y)(1-x-y-xt)(1-x-y-yt)},
\]
\[
\sum_{k=1}^n
f_{j}(\Delta'_{n,k})
=
j \cdot 2^{n-j-1} \frac{n+j+2}{n+1}
\cdot
\binom{n+1}{j+1}.
\]
Finally, in Section $3$, we propose two open questions.

\section{Face numbers of $\Delta'_{n,k}$}

In this section, we study the $f$-vector of $\Delta'_{n,k}$.  

\begin{example}
\label{Boston}
{\em 
Since the hyperplanes of $\Delta'_{n,k}$ are
\begin{eqnarray*}
\HH &=&\{ (x_1,\ldots,x_n)\in [0,1]^n \ : \ x_1+\cdots+x_n=k\},\\
\HH_i^{(0)} &=&\{ (x_1,\ldots,x_n) \in [0,1]^n \ : \  x_i=0 \} \ \ (i=1,\dots,n) ,\label{zero}\\ 
\HH_i^{(1)} &=&\{ (x_1,\ldots,x_n) \in [0,1]^n\ : \  x_i=1 \} \ \ (i=1,\dots,n) ,\label{one}
\end{eqnarray*}
it follows that
$f_{n-1}(\Delta'_{n,k})=2n+1$ for $2 \leq k<n$ 
(and $f_{n-1}  (\Delta'_{n,1}) = n+1$, $f_{n-1}  (\Delta'_{n,n}) = n$).
Let us then see an example of $f_{n-2}(\Delta'_{n,k})$ for $2=k < n$.
This is equivalent to computing pairs of 
hyperplanes $h_1, h_2$ which has $(n-2)$-dimensional intersection with 
$\Delta'_{n,k}$. Here is the enumeration: 
 \begin{enumerate}
 \item $h_1 = \HH_i^{(1)}, \, h_2 =\HH_j^{(1)}$: 
 this kind of pairs do not count, since the intersection with $\Delta'_{n,2}$ is $\{ {\bf e}_i + {\bf e}_j \}$
  which is $0$-dimensional for all $n>2$.
 \item $h_1 = \HH_i^{(1)}, \, h_2 =\HH_j^{(0)}$: 
  this kind of pairs counts for all $n>2$, since the intersection with $\Delta'_{n,2}$ is 
  $$\left\{ \, (x_1,\ldots,x_n) \in [0,1]^n \, : \, x_i = 1,\  x_j = 0, \  0<\sum_{m\in[n]\backslash \{i,j\}}x_m\le 1 \, \right\},$$
  which is $(n-2)$-dimensional for $n>2$.
  Therefore there are $2\binom{n}{2}$ such pairs for $n>2$.
 \item $h_1 = \HH_i^{(0)}, \, h_2 =\HH_j^{(0)}$: 
  this kind of pairs counts for all $n>3$, since the intersection with $\Delta'_{n,2}$ is 
  $$\left\{ \, (x_1,\ldots,x_n) \in [0,1]^n \, : \, x_i = x_j = 0, \  1<\sum_{m\in[n]\backslash \{i,j\}}x_m\le 2 \, \right\},$$
  which is $(n-2)$-dimensional for $n>3$, and 
empty for $n = 3$. 
  Therefore there are $\binom{n}{2}$ such pairs for $n>3$, and $0$ for $n =3$.

 \item $h_1 = \HH, \, h_2 =\HH_i^{(1)}$:  the intersection with $\Delta'_{n,2}$ is 
  $$\left\{ \, (x_1,\ldots,x_n) \in [0,1]^n \, : \, x_i = 1, \ 
\sum_{m\in[n]\backslash \{i\}}x_m= 1 \, \right\},$$
  which is $(n-2)$-dimensional for $n>2$.
 \item $h_1 = \HH, \, h_2 =\HH_i^{(0)}$: 
the intersection with $\Delta'_{n,2}$ is 
  $$\left\{ \, (x_1,\ldots,x_n) \in [0,1]^n \, : \, x_i = 0,\ 
\sum_{m\in[n]\backslash \{i\}}x_m= 2 \, \right\},$$
  which is $(n-2)$-dimensional for $n>3$, and $0$-dimensional for 
$n = 3$.

\end{enumerate}
In conclusion, for $k=2$ and $n\ge 3$, one has
$$
f_{n-2} (\Delta'_{n,2}) =
\left\{
\begin{array}{ll}
2\binom{n}{2}+n & \, \, \, \, \text{if} \, \, \, \, n=3,\\
\smallskip\\
3\binom{n}{2}+2n & \, \, \, \, \text{if} \, \, \, \, n>3.
\end{array}
\right.
$$
(It is easy to see that $f_0 (\Delta'_{2,2}) = 1$.)
}
\end{example}

Following the enumeration method in Example \ref{Boston}, one has the following general formula. 

\begin{theorem}\label{f}
Let $0 < k \leq n$
and $f(\Delta'_{n,k})=(f_{0} (\Delta'_{n,k}), \ldots, f_{n} (\Delta'_{n,k}))$
the $f$-vector of the half-open hypersimplex $\Delta'_{n,k}$.
Then one has $f_0 (\Delta'_{n,k}) = \binom{n}{k}$ and
$$
f_j  (\Delta'_{n,k})
=
\binom{n+1}{j+1}
\sum_{s=\max\{0,k-j\}}^{k-1}
\binom{n-j}{s}
\frac{n-s+1}{n+1}
$$
for $j = 1,2,\ldots,n$.
\end{theorem}

\begin{proof}
 Similar as in Example \ref{Boston} for $f_{n-2}(\Delta'_{n,k})$ when $k=2$, here for general $k$, $f_{n-i}(\Delta'_{n,k})$ is 
 counting the $i$-set $\{h_1,\dots,h_i\}$ which has $(n-i)$-dimensional intersection with $\Delta'_{n,k}$.
 There are again two cases:
in the formula (\ref{below}) below,
 the part of $\binom{n}{i}$ 
deals with the case $\HH$ is not included 
 in these $i$ hyperplanes, and the part of $\binom{n}{i-1}$ is counting the case when $\HH$ is one of the $i$ hyperplanes. 
 
 Let us look at the first case carefully, and the second case is enumerated similarly. In the first case, the $i$-tuple $(h_1,\ldots, h_i)$ satisfies
 $$
h_1 \cap \cdots \cap h_i = 
\left(
\bigcap_{\alpha \in I} \HH_{\alpha}^{(0)}
\right)
\cap 
\left(
\bigcap_{\beta \in J} \HH_{\beta}^{(1)}
\right)
 %
$$  
%
for some
$
I,J\subset [n]
$
such that
$
I\cap J=\emptyset
$
and
$\# (I\cup J)=i
$.
 Let $s=\#J$. Then there are in total $\binom{i}{s}$ such $i$-tuples. Now the key point is whether the intersection of all these $i$ hyperplanes with $\Delta'_{n,k}$
 is $(n-i)$-dimensional or not. Here is their intersection:
 $$
\left\{(x_1,\ldots,x_n)\in [0,1]^n \, : \, 
\begin{array}{c}
x_\alpha = 0 \ (\alpha \in I), \ x_\beta =1 \ (\beta \in J)\\
k-s-1<\sum_{m\in[n]\backslash (I\cup J)}x_m\le k-s
\end{array}
\right\}.$$
 Notice that the intersection is $(n-i)$-dimensional if and only if 
$$1 \le k-s < \# ([n]\backslash ( I\cup J ))=n-i.$$
This is exactly why we have 
$\max\{0,k+i-n\}\le s \le \min\{k-1,i\}$
 in the summand when counting the above $i$-tuples. 
 Thus, we have 
\begin{eqnarray}
f_{n-i}(\Delta'_{n,k})
=\binom{n}{i}
\sum_{s=\max\{0,k+i-n\}}^{\min\{k-1,i\}}\binom{i}{s}+\binom{n}{i-1}
\sum_{s=\max\{0,k+i-n\}}^{\min\{k-1,i-1\}}\binom{i-1}{s}.
\label{below}
\end{eqnarray} 
Hence,
\begin{eqnarray*} 
f_j (\Delta'_{n,k})
 &=&
\binom{n}{j}
\sum_{s=\max\{0,k-j\}}^{k-1}
\binom{n-j}{s}
+\binom{n}{j+1}
\sum_{s=\max\{0,k-j\}}^{k-1}
\binom{n-j-1}{s}\\
 &=&
\sum_{s=\max\{0,k-j\}}^{k-1}
\left(
\binom{n}{j}
\binom{n-j}{s}
+\binom{n}{j+1}
\binom{n-j-1}{s}\right)\\
&=&
\sum_{s=\max\{0,k-j\}}^{k-1}
\binom{n+1}{j+1,s,n-j-s}
\frac{n-s+1}{n+1}\\
&=&
\binom{n+1}{j+1}
\sum_{s=\max\{0,k-j\}}^{k-1}
\binom{n-j}{s}
\frac{n-s+1}{n+1}.
\end{eqnarray*} 
 \end{proof}

\begin{example}
{\em
 Let us take $k=i=2$ for the equation (\ref{below})
in Proof of Theorem \ref{f}. 
 \begin{itemize}
  \item For $n=3$, we have $\max\{0,k+i-n\}=1$ and ${\min\{k-1,i\}}={\min\{k-1,i-1\}}=1$. 
 So $f_{n-2} (\Delta'_{n,k})=\binom{n}{2}\binom{2}{1}+\binom{n}{1}\binom{1}{1}$. 
 \item For $n>3$, we have $\max\{0,k+i-n\}=0$ and ${\min\{k-1,i\}}={\min\{k-1,i-1\}}=1$. 
 So $f_{n-2}(\Delta'_{n,k})=\binom{n}{2}\left(\binom{2}{0}+\binom{2}{1}\right)+\binom{n}{1}\left(\binom{1}{0}+\binom{1}{1}\right)$.
 \end{itemize}
This matches Example \ref{Boston} we computed.
}
\end{example}

Using the above method, it is not hard to get the $f$-vector of the hypersimplex 
$\Delta_{n,k}$
from the $f$-vector of the hypersimplex $\Delta'_{n,k}$.

\begin{corollary}\label{f-}
Let $f(\Delta_{n,k})=(f_{0}(\Delta_{n,k}), \ldots, f_{n}(\Delta_{n,k}))$
be the $f$-vector of the hypersimplex $\Delta_{n,k}$
and $f(\Delta'_{n,k})=(f_{0}(\Delta'_{n,k}), \ldots, f_{n}(\Delta'_{n,k}))$
that of the half-open hypersimplex $\Delta'_{n,k}$.
Then one has $f_0(\Delta_{n,k}) = \binom{n}{k}+\binom{n}{k-1}$
and
$$
f_{j} (\Delta_{n,k})
=
f_{j} (\Delta'_{n,k})+
\binom{n}{j+1}
\sum_{s=\max\{0,k -1-j\}}^{k-2}
\binom{n-j-1}{s}\\
=
\binom{n+1}{j+1}
\sum_{s=\max\{0,k-j\}}^{k-1}
\binom{n-j}{s}
$$
for $j = 1,2,\ldots,n$.
 \end{corollary}
 
\begin{proof}
The only difference with the half-open hypersimplex is that there is now one more case for the $i$-tuple $\{h_1,\dots,h_i\}$, which is when one of 
them is the hyperplane defined by $x_1+\cdots+x_n = k-1$. 
And the set of such $i$-tuples $\{h_1,\dots,h_i\}$ with $(n-i)$-dimensional intersection with $\Delta_{n,k}$ is exactly
the same as the $i$-tuples $\{h_1,\dots,h_i\}$ in the second case of
Proof of Theorem \ref{f} for the half-open hypersimplex $\Delta'_{n,k-1}$, i.e, 
when the hyperplane defined by $x_1+\cdots+x_{n-1}=k-1$
belongs to 
$\{h_1,\dots,h_i\}$. The number of such  $i$-tuples is enumerated by the second summand of the equation (\ref{below}), replacing $k$ by $k-1$. 
Therefore, we obtain the formula
$$f_{n-i} (\Delta_{n,k}) =f_{n-i}(\Delta'_{n,k})+\binom{n}{i-1}
\sum_{s=\max\{0,k-1+i-n\}}^{\min\{k-2,i-1\}}\binom{i-1}{s}.$$
Hence
\begin{eqnarray*}
 & & f_{j} (\Delta_{n,k})\\
&=&
f_{j}(\Delta'_{n,k}) +\binom{n}{j+1}
\sum_{s=\max\{0,k-1-j\}}^{k-2}\binom{n-j-1}{s}\\
&=& \hspace{-0.5cm}
\sum_{s=\max\{0,k-j\}}^{k-1}
\binom{n+1}{j+1,s,n-j-s}
\frac{n-s+1}{n+1}
+
\sum_{s=\max\{0,k -1-j\}}^{k-2}
\binom{n}{j+1, s, n-1-j-s}\\
&=& \hspace{-0.5cm}
\sum_{s=\max\{0,k-j\}}^{k-1}
\binom{n+1}{j+1,s,n-j-s}
\frac{n-s+1}{n+1}
+
\sum_{s=\max\{0,k -j\}}^{k-1}
\binom{n}{j+1, s-1, n-j-s}\\
&=& \hspace{-0.5cm}
\sum_{s=\max\{0,k-j\}}^{k-1}
\left(
\binom{n+1}{j+1,s,n-j-s}
+
\binom{n}{j+1, s-1, n-j-s}
-
\binom{n+1}{j+1,s,n-j-s}
\frac{s}{n+1}
\right)\\
&=& \hspace{-0.5cm}
\sum_{s=\max\{0,k-j\}}^{k-1}
\binom{n+1}{j+1,s,n-j-s}\\
&=&
\binom{n+1}{j+1}
\sum_{s=\max\{0,k-j\}}^{k-1}
\binom{n-j}{s},
\end{eqnarray*}
as desired.
 \end{proof}

\section{Generating functions}
Next we discuss generating functions which are related to
the $f$-vectors of half-open hypersimplices.
 
\begin{theorem}
\label{RS}
Let ${f_j}(\Delta'_{n,k})$ denote the number of $j$-faces of $\Delta'_{n,k}$.
Then, we have
$$
\sum_{n\ge1}
\sum_{k=1}^n
\sum_{j=1}^n
{f_j}(\Delta'_{n,k})\ 
x^k y^{n-k}
t^j
=
\frac{(1-x)xt}{(1-x-y)(1-x-y-xt)(1-x-y-yt)}.
$$
\end{theorem}

\begin{proof}
The coefficient of $x^k y^{n-k} t^j$ in
\begin{eqnarray*}
& &
\frac{tx(1-x)}{(1-x-y)(1-x-y-tx)(1-x-y-ty)}\\
&=&
\frac{tx(1-x)}{(1-x-y)^3} \cdot 
\frac{1}{1-\frac{tx}{1-x-y}} \cdot
\frac{1}{1-\frac{tx}{1-x-y}}\\
&=&
\frac{tx(1-x)}{(1-x-y)^3}
\left(
\sum_{p\ge0}
\left(
\frac{tx}{1-x-y}
\right)^p
\right)
\left(
\sum_{q\ge0}
\left(
\frac{tx}{1-x-y}
\right)^q
\right)\\
&=&
\frac{tx(1-x)}{(1-x-y)^3}
\left(
\sum_{p\ge0}
\sum_{q\ge0}
\left(
\frac{t}{1-x-y}
\right)^{p+q}
x^p y^q 
\right)\\
&=&
\sum_{p\ge0}
\sum_{q\ge0}
\frac{t^{p+q+1}}{(1-x-y)^{p+q+3}}
(1-x) x^{p+1} y^q
\end{eqnarray*}
is equal to the coefficient of $x^k y^{n-k}$ in
$$
\sum_{p=0}^{j-1}
\frac{(1-x) x^{p+1} y^{j-p-1}}{(1-x-y)^{j+2}}
$$
(since $p+q+1=j$ if and only if $q=j-p-1 \geq 0$).
It then follows that
\begin{eqnarray*}
\sum_{p=0}^{j-1}
\frac{(1-x) x^{p+1} y^{j-p-1}}{(1-x-y)^{j+2}}
&=&
\sum_{p=0}^{j-1}
(1-x) x^{p+1} y^{j-p-1}
\sum_{r=0}^\infty
\binom{j+r+1}{j+1}
(x+y)^r\\
&=&
\sum_{p=0}^{j-1}
\sum_{r\ge0}
\sum_{u=0}^r
(1-x) x^{p+1} y^{j-p-1}
\binom{j+r+1}{j+1}
\binom{r}{u}
x^u y^{r-u}\\
&=&
\sum_{p=0}^{j-1}
\sum_{r\ge0}
\sum_{u=0}^r
\binom{j+r+1}{j+1}
\binom{r}{u}
x^{p+u+1} y^{j-p+r-u-1}\\
& &
-
\sum_{p=0}^{j-1}
\sum_{r\ge0}
\sum_{u=0}^r
\binom{j+r+1}{j+1}
\binom{r}{u}
x^{p+u+2} y^{j-p+r-u-1}.
\end{eqnarray*}
Note that
\begin{itemize}
\item
$x^{p+u+1} y^{j-p+r-u-1} = x^k y^{n-k}$ if and only if
$p+u+1=k$ and $j+r=n$ ;
\item
$x^{p+u+2} y^{j-p+r-u-1} = x^k y^{n-k}$ if and only if
$p+u+2=k$ and $j+r+1=n$.
\end{itemize}
Thus, the coefficient of $x^k y^{n-k}$ is
\begin{eqnarray*}
& &
\sum_{p=0}^{j-1}
\sum_{r=n-j}
\sum_{u=k-p-1}
\binom{j+r+1}{j+1}
\binom{r}{u}
-
\sum_{p=0}^{j-1}
\sum_{r=n-j-1}
\sum_{u=k-p-2}
\binom{j+r+1}{j+1}
\binom{r}{u}\\
&=&
\sum_{p=\max\{0,k-1-n+j\}}^{\min\{j-1,k-1\}}
\binom{n+1}{j+1}
\binom{n-j}{k-p-1}
-
\sum_{p=\max\{0,k-1-n+j\}}^{\min\{j-1,k-1\}}
\binom{n}{j+1}
\binom{n-j-1}{k-p-2}\\
&=&
\sum_{s=\max\{0,k-j\}}^{\min\{k-1,n-j\}}
\binom{n+1}{j+1}
\binom{n-j}{s}
-
\sum_{s=\max\{0,k-j\}}^{\min\{k-1,n-j\}}
\binom{n}{j+1}
\binom{n-j-1}{s-1}\\
&=&
\binom{n+1}{j+1}
\sum_{s=\max\{0,k-j\}}^{\min\{k-1,n-j\}}
\left(
\binom{n-j}{s}
-
\frac{n-j}{n+1}
\binom{n-j-1}{s-1}
\right)\\
&=&
\binom{n+1}{j+1}
\sum_{s=\max\{0,k-j\}}^{\min\{k-1,n-j\}}
\binom{n-j}{s}
\frac{n-s+1}{n+1},
\end{eqnarray*}
as desired.
\end{proof}

By the proof of Theorem \ref{RS}, it is also shown that

\begin{corollary}
Let $f_j (\Delta_{n,k})$ denote the number of $j$-faces of $\Delta_{n,k}$.
Then, we have
$$
\sum_{n\ge1}
\sum_{k=1}^n
\sum_{j=1}^n
f_j (\Delta_{n,k})\ 
x^k y^{n-k}
t^j
=
\frac{xt}{(1-x-y)(1-x-y-xt)(1-x-y-yt)}.
$$
\end{corollary}

On the other hand, by Theorem \ref{RS}, we have
the 
$f$-vector of the hypersimplicial decomposition of the unit cube,
and its generating function.

\begin{theorem}\label{sum}
Let ${f_j} (\Delta'_{n,k})$ denote the number of $j$-faces of
the half-open hypersimplex $\Delta'_{n,k}$.
Then, we have
$$
\sum_{k=1}^n
{f_j} (\Delta'_{n,k})
=
j \cdot 2^{n-j-1} \frac{n+j+2}{n+1}
\cdot
\binom{n+1}{j+1}
$$
and
$$
\sum_{n\ge1}
\sum_{k=1}^n
{f_j} (\Delta'_{n,k})\ 
x^n
=
\frac{j x^j(1-x)}{(1-2x)^{j+2}}
.$$
\end{theorem}

\begin{proof}
By substituting $x$ for $y$ in the equation of Theorem \ref{RS}, we have
\begin{eqnarray*}
\sum_{n\ge1}
\sum_{k=1}^n
\sum_{j=1}^n
{f_j} (\Delta'_{n,k})\ 
x^n
t^j
&=&
\frac{(1-x)xt}{(1-2x)(1-2x-xt)^2}\\
&=&
\frac{(1-x)xt}{(1-2x)^3(1-\frac{x}{1-2x} t)^2}\\
&=&
\frac{(1-x)xt}{(1-2x)^3}
\sum_{j\ge 0} (j+1) \left(\frac{x}{1-2x}\right)^j t^j.
\end{eqnarray*}
Thus, we have
$$
\sum_{n\ge 1}
\sum_{k=1}^n
{f_j} (\Delta'_{n,k})
x^n
=
\frac{x(1-x)}{(1-2x)^3} \cdot
j \left(\frac{x}{1-2x}\right)^{j-1}
=
\frac{j x^j(1-x)}{(1-2x)^{j+2}}
.$$
Moreover, the coefficient of $x^n$ in 
\begin{eqnarray*}
\frac{j x^j(1-x)}{(1-2x)^{j+2}}
&=&
j x^j(1-x) \sum_{m\ge 0}
\binom{m+j+1}{j+1} 2^m x^m\\
&=&
j
\sum_{m\ge 0}
\binom{m+j+1}{j+1} 2^m x^{m+j}
-
j
\sum_{m\ge 0}
\binom{m+j+1}{j+1} 2^m x^{m+j+1}\\
&=&
j
\sum_{m\ge 0}
\binom{m+j+1}{j+1} 2^m x^{m+j}
-
j
\sum_{m\ge 0}
\binom{m+j}{j+1} 2^{m-1} x^{m+j}\\
&=&
j
\sum_{m\ge 0}
\left(
2 \binom{m+j+1}{j+1}
-
\binom{m+j}{j+1} 
\right)
2^{m-1} x^{m+j}
\\
&=&
j
\sum_{m\ge 0}
\frac{m+2j+2}{m+j+1}
\binom{m+j+1}{j+1} 
2^{m-1} x^{m+j}
\end{eqnarray*}
is
$$
j
\frac{(n-j)+2j+2}{(n-j)+j+1}
\binom{(n-j)+j+1}{j+1} 
2^{(n-j)-1} 
=
j \cdot 2^{n-j-1} \frac{n+j+2}{n+1}
\cdot
\binom{n+1}{j+1}.$$
\end{proof}

\begin{proof}[Geometric proof of Theorem \ref{sum}]
The first equation of Theorem \ref{sum} gives 
the $f$-vector of the hypersimplicial decomposition of the unit cube. 
Here we compute this $f$-vector directly. Similar as the proof of 
Theorem \ref{f}, for ${f_j} (\Delta'_{n,k})$, we count the $j$-dimensional intersections obtained by $(n-j)$ hyperplanes of the unit cube. 
There are
again two types of such $j$-faces:
\begin{enumerate}
 \item obtained by intersections of $n-j$ hyperplanes of the form 
$\HH_i^{(0)}$ or $\HH_j^{(1)}$.
There are $2^{n-j}\binom{n}{n-j}$;
 \item obtained by intersections of $n-j-1$ hyperplanes of the form
$\HH_i^{(0)}$ or $\HH_j^{(1)}$
together with one more hyperplane
defined by
 $\sum_{i=1}^nx_i =k$. 
Similar to the argument in the proof of Theorem \ref{f}, we can see that each combination of the 
 $n-j-1$ hyperplanes intersects nontrivially (i.e., get $j$-dimensional intersection) with exactly $j$ hyperplanes of 
 the form $\sum_{i=1}^nx_i =k$. In fact, for a given choice of $n-j-1$ hyperplanes of the form $\HH_i^{(0)}$ or $\HH_j^{(1)}$, let $s$ be the number 
 of indices $m$ with $\HH_m^{(1)}$, then above $j$ hyperplanes corresponds to $k=s+1,\dots,s+j$. 
Therefore, there are $j\cdot2^{n-j-1}\binom{n}{n-j-1}$
 such $j$-faces.
\end{enumerate}

Now comes a tricky part: the correct number of first type $j$-face should be $2^{n-j}\binom{n}{n-j}$ times $j$. This is because 
there are exactly $(j-1)$ $(j-1)$-faces in the interior of each such $j$-face, resulting in $j$ $j$-faces. Consider $f_2$ for the three dimensional cube
for a visual help. Therefore, there are in total 
$$2^{n-j}\binom{n}{n-j}\cdot j + j\cdot2^{n-j-1}\binom{n}{n-j-1} = j \cdot 2^{n-j-1} \frac{n+j+2}{n+1}
\cdot
\binom{n+1}{j+1}$$
$j$-faces in the hypersimplicial decomposition of the unit cube. 
\end{proof}

\section{Two open questions}

In this section, we present two open questions related with Theorem \ref{sum}. 

First, notice that in the above geometric proof for Theorem \ref{sum}, we get the result as a sum of two parts. Since the result of Theorem \ref{sum} is neat and simple,
it would be very nice to have a direct combinatorial proof avoiding sums.
\begin{question}
{\em Find a combinatorial proof for Theorem \ref{sum}.
}
\end{question}

We also observe a relation between Theorem \ref{sum} and the coefficients of Chebyshev polynomials.
It is known that, for $\ell >0$,
the Chebyshev polynomials 
$
T_\ell(x)$
of the first kind satisfies
\begin{eqnarray*}
T_\ell(x)
&=&
\frac{\ell}{2} \sum_{m=0}^{\lfloor \frac{\ell}{2} \rfloor}
(-1)^m
\frac{(\ell -m -1)!}{m!(\ell -2m)!}  (2 x)^{\ell -2m}\\
&=&
\sum_{m=0}^{\lfloor \frac{\ell}{2} \rfloor}
(-1)^m 
\frac{\ell}{\ell -m}
\binom{\ell-m}{m}
  2^{\ell -2m-1} x^{\ell -2m}
\end{eqnarray*}
On the other hand, 
we proved that
$$
\frac{1}{j}
\sum_{k=1}^n
{f_{j}} (\Delta'_{n,k})
=
2^{n-j-1}
\frac{n+j+2}{n+1}
\binom{n+1}{j+1}
.$$
Thus, for $j =m-1$ and $n=\ell-m-1$,
we have
$$
\frac{1}{j}
\sum_{k=1}^n
{f_{j}} (\Delta'_{n,k})
=
2^{\ell -2m-1}
\frac{\ell}{\ell-m}
\binom{\ell -m}{m}
$$
which is the absolute value of
the coefficient of $x^{\ell-2m}$
in $T_\ell(x)$. There are some known combinatorial models for Chebyshev polynomials such as \cite{BW}, 
 \cite{S} and \cite{M}, but their connection with $f$-vectors
studied here is not clear to us.
\begin{question}
{\em Find a combinatorial connection between Chebyshev polynomials and the sum of $f$-vectors 
for the half-open hypersimplex, or equivalently, the $f$-vectors of the hypersimplicial decomposition of the unit cube.

}
\end{question}

\noindent
\begin{acknowledgement}
 We want to thank Professor Richard Stanley for very helpful comments on our work.
\end{acknowledgement}

\end{document}